\documentclass{ifacconf}

\usepackage{graphicx}      
\usepackage{natbib}        
\usepackage{pstricks}
\usepackage{amsmath,amssymb,amsopn,cases}
\usepackage{setspace}

\newcommand{\supeq}{\geqslant}
\newcommand{\infeq}{\leqslant}

\newtheorem{de}{Definition}
\newtheorem{theo}{Theorem}
\newtheorem{proposition}{Proposition}
\newtheorem{coro}{Corollary}
\newtheorem{lemma}{Lemma}
\newtheorem{remark}{Remark}
\newenvironment{proof}[1][]{\textbf{Proof #1:~}}{\hfill$\square$\\}

\begin{document}

\begin{frontmatter}

\title{Wirtinger-based Exponential Stability for Time-Delay Systems\thanksref{footnoteinfo}} 

\date{\today}  

\author[First]{Matthieu Barreau} 
\author[First]{Alexandre Seuret} 
\author[First]{Fr{\'e}d{\'e}ric Gouaisbaut} 

\thanks[footnoteinfo]{This work is supported by the ANR project SCIDiS contract number 15-CE23-0014.}

\address[First]{LAAS-CNRS, Universit\'e de Toulouse, CNRS, UPS, Toulouse, France. (e-mail: barreau,seuret,fgouaisb@laas.fr).
}

\begin{abstract}
This paper deals with the exponential stabilization of a time-delay system with an average of the state as the output. A general stability theorem with a guaranteed exponential decay-rate based on a Wirtinger-based inequality is provided. Variations of this theorem for synthesis of a controller or for an observer-based control is derived. Some numerical comparisons are proposed with existing theorems of the literature and comparable results are obtained but with an extension to stabilization.
\end{abstract}


\end{frontmatter}

\section{Introduction}

Time-delay systems may arise in practice for many reasons. For example, it appears in mechanical modeling like vibration absorber (see \cite{OLGAC199493}) or delayed resonator (see \cite{opac-b1100602}) which are intrinsically with delay and neglecting it leads to an over-simplification of the initial problem. That is why it is important to have a theory which can provide a framework to work with. Indeed, although time-delay systems are a class of dynamical systems widely studied in control theory, the honored method like root-locus to assess stability are not straightforward, particularly to provide robust stability criteria.


Three main approaches have been developed to study the stability of the such equations. The first one relies on the characteristic equation (see \cite{5687820} and references therein and \cite{BREDA2006305}) and pole location. These techniques give nearly the exact stability conditions but suffer from several drawbacks. First of all, as they are based on pole location approximations, they are not appropriated for uncertain and/or time-varying delay systems. Furthermore, these approaches could not also be used easily for the design of controllers or observers. 

Other approaches have been developed based either on the robust approach or Lyapunov techniques. The robust approach consists of merging the delay uncertainty into an uncertain set and use classical robust analysis as Small Gain Theorem (\cite{fridman2008input}), Quadratic Separation (\cite{gouaisbaut2006delay}), Integral Quadratic Constraints (\cite{kao2007stability}). Techniques based on Lyapunov-Krasovskii functionals uses the LMI framework developed in the book by \cite{LMI}. This method enables exponential convergence with a guaranteed decay rate, robust analysis, synthesis of controllers and extension to multiple time-varying delay systems. 

Despite these advantages, this approach is very conservative. The complete Lyapunov-Krasovskii functional is known (\cite{kharitonov2003lyapunov}) but too complex to be efficiently solved and even studied. A first step is to introduce a simplified functional. Some works have been done (for example by \cite{seuret:hal-01065142}) on how to relax the problem such that the conservatism introduced by the choice of the Lyapunov-Krasovskii functional is measured. The second step is to use integral inequalities to transform some non-manageable terms like $\int_{t-h}^t e^{-2\alpha s} x^{\top}(t+s) R x(t+s)ds$ into an expression suitable to be transformed into LMIs. This last step is important because there exists powerful and efficient algorithm to find solutions of LMIs in polynomial time. The commonly used inequalities in the two last steps are described by \cite{opac-b1100602} and rely for most of them on Jensen's inequality.  An important amount of papers have been dedicated to reduce the conservatism induced by such inequalities. Recently, \cite{wirtinger} introduced a Wirtinger-based inequality, known to be less conservative. The present paper uses this framework to state the exponential convergence with a guaranteed decay rate and synthesis of controllers.

Two approaches have been widely used in the literature to assess the exponential stability. The first one relies on a change of variable $z(t) = e^{\alpha t}x(t)$ and it can be proven that establishing asymptotic stability of $z$ implies an exponential stability of $x$ with a decay rate of $\alpha$ (\cite{tds}). The second one is based on some modified Lyapunov-Krasovskii functionals which incorporate in their structures the exponential rate. 

Since one of the first article by \cite{mori1982estimate} on exponential convergence of time-delay systems, several exponential estimates emerged from the literature: \cite{mondie2005exponential}, \cite{lam} or more recently \cite{trinh2016exponential}. But only a few of them used the Wirtinger-based inequality developed by \cite{wirtinger} to help synthesize observers or controllers for a discrete or distributed delay system. The aim of this article is to stabilize a specific class of time-delay systems as described in the problem statement using this inequality.

In Section 2, the problem is stated and some useful lemmas are reminded. Then in Section 3, an extension of exponential stability theorems with a Wirtinger-based inequality is introduced. The general results of the previous section are used for the computation of a feedback gain for a given system in Section 4 while Section 5 is dedicated to the design of an observer-based control. Finally, in the last section, a numerical comparison of efficiency between classical theorems and the one derived in this paper is performed.

\textbf{Notations.} Throughout the paper, $\mathbb{R}^n$ stands for the $n$ dimensional Euclidian space, $\mathbb{R}^{n \times m}$ for the set of all $n \times m$ matrices. $\mathbb{S}^n$ is the subset of $\mathbb{R}^{n \times n}$ of symmetric matrices such that $P \in \mathbb{S}_+^n$ or equivalently $P \succ 0$ denotes a symmetric positive definite matrix.  For any square matrices $A$ and $B$, the operations '$\text{He}$' and '$\text{diag}$' are defined as follow: $\text{He}(A) = A + A^{\top}$ and $\text{diag}(A,B) = \left[ \begin{smallmatrix}A & 0\\ 0 & B \end{smallmatrix} \right]$. The notations $I_n$ and $0_{n \times m}$ denote the $n$ by $n$ identity matrix and the null matrix of size $n  \times m$. The state variable $x$ can be represented using the Shimanov notation (\cite{kolmanovskii}):
$
	x_t: \left\{ \begin{array}{rccl}
		& [-h, 0] & \to & \mathbb{R}^n\\
		& \tau    &  \mapsto & x(t+\tau)
	\end{array}
	\right.
$

\section{Problem Statement}

\subsection{System data}

%
The system to be controlled is the following one:
\begin{equation}
	\left\{
		\begin{array}{lcl}
			\dot{x}(t) = Ax(t) + Bu(t), & \ & \forall t \supeq 0, \\
			\displaystyle y(t) = C \frac{1}{h} \int_{-h}^0 x_t(s) ds, & \ & \forall t \supeq 0, \\
			x(t) = \phi(t), & \ & \forall t \in [-h, 0],
		\end{array}
	\right.
	\label{eq:sys}
\end{equation}
with $x(t) \in \mathbb{R}^n$ the instantaneous state vector, $h$ the time delay, $\phi$ the initial state function and $A$, $B$, $C$ three matrices of appropriate dimensions. Then, the output is not the instantaneous state but its average on a sliding window of time $[t-h, t]$, which differs significantly from classical control problems. Numerous measurement tools, in electronics for example, are measuring an average and not the instantaneous state.

The purpose of this paper is to find a control input $u$ computed only with the output measurement vector $y$ such that System \eqref{eq:sys} is exponentially stable with a decay rate of at least $\alpha \supeq 0$. First of all, we recall the definition of exponential stability extended to time-delay systems:
\begin{de}
	[\cite{Chen200795}] System \eqref{eq:sys} is said to be $\alpha$-stable if there exists $\alpha \supeq 0$ and $\gamma \supeq 1$ such that for every solution $x$ of \eqref{eq:sys} with a differentiable initial condition $\phi$ defined on $[-h; 0]$, the following exponential estimate holds:
	\begin{equation}	
		\forall t > 0, \left| x(t) \right| \infeq \gamma e^{-\alpha t} \left\lVert \phi \right\rVert_W 
		\label{eq:expoConvergence}
	\end{equation}
	where 
	$$\left\lVert \phi \right\rVert_W = \max\{ ||\phi||_h, ||\dot{\phi}||_h\} \text{ and } \displaystyle \left\lVert \phi \right\rVert_h = \sup_{\theta \in [-h, 0]} \left\lVert \phi(\theta) \right\rVert. $$
\end{de}

\begin{remark}
	The norm $\lVert \cdot \rVert_W$ is sightly different from the one of \cite{mondie2005exponential} who do not consider a norm depending on the derivative $\dot{x}$. This problem has also been dealt by \cite{norm} by introducing the sum and not the maximum. These definitions are nevertheless equivalent.
\end{remark}

\subsection{Preliminary Results}

We recall two lemmas useful in the sequel. The first lemma, introduced by \cite{wirtinger} proposes an integral inequality which is used in the proof of the main theorem. 
\begin{lemma} [Wirtinger-based inequality] For a given matrix $R \in \mathbb{S}^n_+$, the following inequality holds for all continuously differentiable function $x$ in $[t-h, t] \to \mathbb{R}^n$:
	\begin{equation*}
		\int_{t-h}^t \dot{x}^{\top}(s) R \dot{x}(s) ds \supeq \frac{1}{h} \xi^{\top}(t) F_2^{\top} \tilde{R} F_2 \xi(t),
	\end{equation*}
	where 
	\[
		\begin{array}{ccc}
			F_2 = \left[ \begin{matrix} I_n & -I_n & 0_n \\ I_n & I_n &  -2I_n \end{matrix} \right], & \ \ \ \ \ &
			\tilde{R} = \text{diag}\left( R, 3R \right), \\
		\end{array} \\
	\]
	\[
		\xi(t) = \left[ \begin{matrix} x^{\top}(t) & \ x^{\top}(t-h) \ & \frac{1}{h} \int_{t-h}^t x^{\top}(s)ds \end{matrix} \right]^{\top}.
	\]
	\label{sec:wirtinger}
\end{lemma}

The second lemma, called Finsler's lemma, is widely used to cope with non linearities in LMIs.
\begin{lemma} (\cite{Svariable}) \\
	For any $Q \in \mathbb{S}^n$ and $M \in \mathbb{R}^{p \times n}$, the three following properties are equivalent:
	\begin{enumerate}	
		\item $x^{\top} Q x \prec 0$ for all $x \in \mathbb{R}^n \text{ such that } Mx = 0$,
		\item $\exists Y \in \mathbb{R}^{n \times p}, Q + \text{He} \left( M^{\top} Y \right) \prec 0$,
		\item ${M^{\perp}}^{\top} Q M^{\perp} \prec 0$ where $MM^{\perp} = 0$.
	\end{enumerate}
	\label{sec:finsler}
\end{lemma}

\section{Exponential Stability}

Considering a feedback on System \eqref{eq:sys}, i.e. $u(t) = Ky(t)$, it is possible to transform our system into a more general one:
\begin{equation}
	\left\{
		\begin{array}{ll}
			\displaystyle \dot{x}(t) = Ax(t) + A_d x_t(-h) + A_D\int_{-h}^0 x_t(s) ds, & \forall t \supeq 0, \\
			x(t) = \phi(t), \hfill \forall t & \in [-h, 0],
		\end{array}
	\right.
	\label{eq:sysDelayed}
\end{equation}
with $x(t) \in \mathbb{R}^n$ the instantaneous state vector and matrices $A$, $A_d$ and $A_D$ of appropriate dimensions. 

Based on the lemmas recalled above, we propose a first exponential stability result for the previous system.
\begin{theo}
Assume that, for given $h > 0$ and $\alpha \supeq 0$, there exist matrices $P \in \mathbb{S}^{2n}$,  $R, S \in \mathbb{S}^n_+$ and $Y \in \mathbb{R}^{n \times 4n}$ and a positive real $\beta_1$ such that the following LMIs are satisfied:
	\begin{equation}
		\begin{array}{l}
			P + \tfrac{e^{-2 \alpha h}}{h} \text{diag}(0_n, S) + \tfrac{4 \alpha^2 h}{e^{2 \alpha h}-2h\alpha-1} \left[ \begin{smallmatrix} h^2R & -hR \\ -hR & R \end{smallmatrix} \right] \ \ \ \ \ \ \\
			 \hfill -  \beta_1 \text{diag}\left( I_n, 0_n \right) \succ 0, 
		\end{array}
		\label{eq:positivity}
	\end{equation}
	\begin{equation}	
		\Phi(\alpha, h) + \text{He}\left(  F_4^{\top} Y \right) \prec 0,
	\end{equation}
	with
	\begin{equation*}
		\begin{array}{rcl}
			\Phi(\alpha, h) & = & \text{He}\left( F_1^{\top} P (F_0 + \alpha F_1) \right) + \bar{S} + h^2 F_3^{\top} R F_3 \\
			& &  - e^{-2 \alpha h} F_2^{\top} \tilde{R} F_2 , \\
		\end{array}
	\end{equation*}
	\begin{equation*}
	\begin{array}{ll}
		F_0 = \left[ \begin{matrix}0_n & 0_n & I_n & 0_n \\ I_n & -I_n & 0_n & 0_n \end{matrix} \right], &
		F_1(h) = \left[ \begin{matrix} I_n & 0_n & 0_n & 0_n \\ 0_n & 0_n & 0_n & hI_n \end{matrix} \right], \\
		F_2 = \left[ \begin{matrix} I_n & -I_n & 0_n & 0_n \\ I_n & I_n & 0_n & -2I_n \end{matrix} \right] ,
		& 
		\begin{array}{l}
			F_3 = \left[ \begin{matrix} 0_n & 0_n & I_n & 0_n \end{matrix} \right],\\
			F_4 = \left[ \begin{matrix} A & A_d & -I_n & hA_D \end{matrix} \right],
		\end{array} \\
		\tilde{R} = \text{diag} \left( R, 3R \right), \ & \ \bar{S} = \text{diag} \left( S, - e^{-2 \alpha h} S, 0_{2n}\right),
	\end{array}
	\end{equation*}
then, time-delay system \eqref{eq:sysDelayed} is $\alpha$-exponentially stable i.e.:
\begin{equation*}
	|| x(t) || \infeq \sqrt{\beta_2 \beta_1^{-1}}  e^{-\alpha t} ||\phi ||_W,
\end{equation*}
where $\beta_2 = \displaystyle (1 +  h^2) \lambda_{max}(P) +  h \lambda_{max}(S) + \frac{h^3}{2} \lambda_{max}(R)$.
\label{sec:thm}
\end{theo}
\begin{proof}
This proof is divided into two parts. 
	
	{\em Part 1: Stability of system \eqref{eq:sysDelayed}} \
	Consider a slightly modified Lyapunov-Krasovskii functional originally proposed by \cite{wirtinger,mondie2005exponential}:
	\begin{equation}
		\begin{array}{lcl}
			 V(x_t, \dot{x}_t)  & = & \displaystyle \bar{x}^{\top}(t) P \bar{x}(t) + \int_{t-h}^t e^{-2\alpha(t-s)} x^{\top}(s) S x(s) ds \ \ \\
			 & & \hfill \displaystyle + h\int_{t-h}^t \int_{\theta}^t e^{-2\alpha(t-s)} \dot{x}^{\top}(s) R \dot{x}(s) ds d\theta,
		\end{array}
		\label{eq:lyap2}
	\end{equation}
	with the extended state $\bar{x}(t) = \displaystyle \left[ x^{\top}(t) \ \ \int_{t-h}^t x^{\top}(s) ds \right]^{\top}$. 
	

	Let us firstly introduce functional $W_{\alpha}$ given by:
	\[
		W_{\alpha}(x_t, \dot{x}_t) = \dot{V}(x_t, \dot{x}_t) + 2 \alpha V(x_t, \dot{x}_t)
	\]
	We want to find an LMI condition so that inequality:
	\begin{equation}
		W_{\alpha}(x_t, \dot{x}_t) < 0,
		\label{eq:newLMI}
	\end{equation}
	is guaranteed for system \eqref{eq:sysDelayed}.

The derivative of functional \eqref{eq:lyap2} along the trajectories of time-delay system \eqref{eq:sysDelayed} leads to:
%
\begin{equation}
	\begin{array}{lcl}
		 W_{\alpha}(x_t, \dot{x}_t) & \infeq &\displaystyle \dot{\bar{x}}^{\top}(t) P \bar{x}(t) + \bar{x}^{\top}(t) P \dot{\bar{x}}(t) + 2 \alpha \bar{x}^{\top}(t) P \bar{x}(t) \\
		 & &\displaystyle + x^{\top}(t) S x(t) - e^{-2 \alpha h} x^{\top}(t-h) S x(t-h) \\
		 & &\displaystyle + h^2 \dot{x}^{\top}(t) R \dot{x}(t) - h e^{-2\alpha h} \int_{t-h}^t  \hspace{-0.2cm} \dot{x}^{\top}(s) R \dot{x}(s) ds
	\end{array}
	\label{eq:inequality1}
\end{equation}

Using the extended state variable
\[
	\xi(t) = \left[ \begin{array}{cccc}\displaystyle x^{\top}(t) \ \ \ x^{\top}(t-h) \ \ \ \dot{x}^{\top}(t) \ \ \ \frac{1}{h} \int_{t-h}^t x^{\top}(s) ds \end{array}\right]^{\top},
\]
and the matrices defined in this theorem, inequality \eqref{eq:inequality1} can be rewritten as:
\begin{equation*}
	\begin{array}{ll}
		W_{\alpha}(x_t, \dot{x}_t) \infeq & \displaystyle\xi^{\top}(t) \left[ \text{He}\left( F_1^{\top} P ( F_0 + \alpha F_1\right)  + \bar{S} \right. \\
		& \displaystyle \left. +h^2 F_3^{\top} R F_3 \right] \xi(t)  - h e^{-2\alpha h} \hspace{-0.1cm}\int_{t-h}^t \hspace{-0.2cm} \dot{x}^{\top}(s) R \dot{x}(s) ds.
	\end{array}
	\label{eq:finsler1}
\end{equation*}

Then, using the integral inequality from Lemma \ref{sec:wirtinger}, we obtain:
\begin{equation*}
	\begin{array}{ll}
		W_{\alpha}(x_t, \dot{x}_t) \infeq & \displaystyle \xi^{\top}(t) \left[ \text{He}\left( F_1^{\top} P (F_0 + \alpha F_1) \right)  + \bar{S}  \right. \\
		& \displaystyle \left. + h^2 F_3^{\top} R F_3 - e^{-2 \alpha h} F_2^{\top} \tilde{R} F_2 \right] \xi(t),
	\end{array}
	\label{eq:finsler2}
\end{equation*}
where $\xi$ satisfies a linear constraint defined by $F_4 \xi = 0$. Therefore, using Lemma \ref{sec:finsler}, $\xi^{\top} \Phi(\alpha, h) \xi \infeq 0$ with $F_4 \xi = 0$, inequality \eqref{eq:newLMI} is satisfied if the following LMI is also satisfied:
\begin{equation}
	\displaystyle \exists Y \in \mathbb{R}^{n \times 4n}, \Phi(\alpha, h) + \text{He}\left( F_4^{\top} Y \right) \prec 0,
	\label{eq:finsler3}
\end{equation}
which concludes the first part of the proof.

{\em Part 2 Exponential stability}: The proof of exponential stability is based on inequality \eqref{eq:finsler3}. Indeed, as it has been noticed by \cite{mondie2005exponential}, the inequality \eqref{eq:finsler3} leads to:
\begin{equation}
	 V(x_t, \dot{x}_t) \infeq e^{-2\alpha t} V(\phi, \dot{\phi}),
	 \label{eq:expConv}
\end{equation}

To ensure the exponential stability of system \eqref{eq:sysDelayed}, one should find strictly positive reals $\beta_1$ and $\beta_2$ such that:
\begin{equation}
	\beta_1 || x(t) ||^2 \infeq V(x_t, \dot{x}_t) \infeq \beta_2 ||x_t||^2_W
	\label{eq:lyapIneq}
\end{equation}

A lower bound for equation \eqref{eq:lyap2} can be derived using Jensen inequality and the inequality derived in appendix \ref{sec:app1}. The Bessel-like inequality developed in appendix \ref{sec:app1} is similar to Jensen's inequality but deals with the exponential terms.
\begin{equation*}
	\begin{array}{lcl}
		V(x_t, \dot{x}_t) & \supeq & \displaystyle \bar{x}^{\top}(t) P \bar{x}(t) \\
		& & + h \int_{t-h}^t \int_{\theta}^t e^{-2\alpha(t-s)} \dot{x}^{\top}(s) R \dot{x}(s) ds d\theta\\
		& & \displaystyle + \frac{e^{-2 \alpha h}}{h} \left( \int_{t-h}^t x^{\top}(s) ds \right) S \left( \int_{t-h}^t x(s) ds \right).\\
	\end{array}
 \end{equation*}
 
 Then, by Jensen's inequality, we have:
 \begin{equation*}
 	\begin{array}{lcl}
 		V(x_t, \dot{x}_t) & \displaystyle \supeq & \bar{x}^{\top}(t) \left( P + \frac{e^{-2 \alpha h}}{h} \text{diag}(0, S) \right. \\
 		& & \hfill \displaystyle + \tfrac{4 \alpha^2 h}{e^{2 \alpha h}-2h\alpha-1} \left[ \begin{smallmatrix} h^2R & -hR \\ -hR & R \end{smallmatrix} \right]  \\
 		& & \hfill \left. -  \beta_1 \text{diag}\left( I_n, 0_n \right) \vphantom{\frac{e^{-2 \alpha h}}{h}} \right) \bar{x}^{\top} + \beta_1 || x(t) ||^2. \\
	\end{array}
 \end{equation*}
 
 Assuming LMI \eqref{eq:positivity} holds, then the previous equation becomes:
 \begin{equation}
 	V(x_t, \dot{x}_t) \supeq \beta_1 || x(t) ||^2.
 \end{equation}

Using equation \eqref{eq:expConv} and \eqref{eq:lyapIneq}, one can get:
\begin{equation*}
	\beta_1 || x(t) ||^2 \infeq V(x_t, \dot{x}_t) \infeq e^{-2\alpha t} V(\phi, \dot{\phi}).
\end{equation*}

Calculating $V(\phi, \dot{\phi})$, one can get the following upper bound: 
\begin{equation*}
	\begin{array}{lcl}
		 V(\phi, \dot{\phi})  & = & \displaystyle \bar{\phi}^{\top}(0) P \bar{\phi}(0) + \int_{-h}^0 e^{2 \alpha s} \phi^{\top}(s) S \phi(s) ds \\
		 & & \displaystyle + h \int_{-h}^0 \int_{\theta}^0 e^{2 \alpha s} \dot{\phi}^{\top}(s) R \dot{\phi}(s) ds d\theta, \\
	\end{array}
\end{equation*}
with $\bar{\phi}(0) = \displaystyle \left[ \phi^{\top}(0) \ \ \int_{-h}^0 \phi^{\top}(s) ds \right]^{\top}$. We get:
\begin{equation*}
	\begin{array}{ccl}
		 V(\phi, \dot{\phi})  & \infeq & \displaystyle \left( (1 + h^2) \lambda_{max}(P) +  h \lambda_{max}(S) \right) ||\phi||_h^2 \\
		 &&+ \frac{h^3}{2} \lambda_{max}(R) ||\dot{\phi}||_h^2\\
		 & \infeq & \displaystyle \beta_2 ||\phi ||^2_W, \\
	\end{array}
\end{equation*}
with 
\[
	\beta_2 = \displaystyle (1+ h^2) \lambda_{max}(P) +  h \lambda_{max}(S) + \frac{h^3}{2} \lambda_{max}(R).
\]

Using the previous equation and \eqref{eq:lyapIneq}, one can get:
\begin{equation*}
	\beta_1 || x(t) ||^2 \infeq V(x_t, \dot{x}_t) \infeq e^{-2\alpha t} V(\phi, \dot{\phi}) \infeq \beta_2 e^{-2\alpha t} ||\phi ||^2_W
\end{equation*}
which is the same than:
\begin{equation*}
	|| x(t) || \infeq \underbrace{\sqrt{\beta_2 \beta_1^{-1}}}_{\gamma}  e^{-\alpha t} ||\phi ||_W,
\end{equation*}
and that concludes the proof.
\end{proof}

\begin{remark}
	The lower bound $\beta_1$ has been explicitly stated such that an optimization of $\gamma$ should be possible. 
 \end{remark}
 \begin{remark}
 	By fixing $\alpha = 0$, one can recover the case of asymptotic stability developed by \cite{wirtinger}.
\end{remark}


\begin{remark}
At the light of Lemma \ref{sec:finsler} proposition 3, using slack variables is not mandatory and is useless for analysis purposes. Nevertheless, we will show that it is suitable for design purposes.
\end{remark}

\begin{coro}
	Assume that, for given $h > 0$ and $\varepsilon_1, \varepsilon_2, \varepsilon_3, \varepsilon_4$ in $\mathbb{R}$, $\alpha \supeq 0$, there exist matrices $P \in \mathbb{S}^{2n}$,  $R, S \in \mathbb{S}^n_+$ and $Z \in \mathbb{R}^{n \times n}$ and a positive real $\beta_1$ such that the positivity LMI \eqref{eq:positivity} and the following LMI are satisfied:
	\begin{equation}
		\Phi(\alpha, h) + \text{He}\left(  F_4^{\top} Z F_{\varepsilon} \right) \prec 0, \\
		\label{eq:LMIcor}
	\end{equation}
	with
	\begin{equation*}
		F_{\varepsilon} = \left[ \begin{matrix} \varepsilon_1 I_n & \varepsilon_2 I_n & \varepsilon_3 I_n & \varepsilon_4 I_n \end{matrix} \right],
		\label{eq:cor}
	\end{equation*}
	Then system \eqref{eq:sysDelayed} is $\alpha$-stable and $Z$ is not singular.
	\label{sec:cor}
\end{coro}
\begin{proof}
	Applying Theorem \ref{sec:thm}, with $Y = ZF_{\varepsilon}$ leads to this result. LMI \eqref{eq:LMIcor} leads to the result $-\varepsilon_3(Z^{\top} + Z) \prec 0$ which means $\varepsilon_3 \neq 0$ and $Z$ is not singular. This proof is constraining $Y$ so this is not equivalent to the previous theorem. The Finsler's lemma can be seen as assessing the stability of two systems at the same time. Considering $Y = ZF_{\varepsilon}$, and by applying Finsler's lemma on equation \eqref{eq:finsler3} with $F_4$ the vector of slack variables, that leads to the stability of another system:
	\[
		\varepsilon_3 \dot{x}(t) = -\varepsilon_1 x(t) - \varepsilon_2 x(t-h) - \varepsilon_4 \frac{1}{h} \int_{-h}^0 x(t+s) dx
	\]
	There are then two possible choices for $F_{\varepsilon}$:
		\begin{enumerate}
			\item $F_{\varepsilon \not = 1}$ is $\varepsilon_3 = \varepsilon_1 = \varepsilon_4 = 1$ and $\varepsilon_2 = 0$
			\item$F_{\varepsilon = 1}$ is $\varepsilon_3 = \varepsilon_1 = \varepsilon_4 = \varepsilon_2 = 1$
		\end{enumerate}
	The first choice sees the delayed term $x(t-h)$ as a perturbation. Perhaps, deleting the effect of this term would stabilize the system, that means $\varepsilon_2 = 0$. The other choice considers that the delayed term is helping the stabilization of the system. The two choices are confronted in numerical simulations later on. 
\end{proof}

\section{Control Design}

In this part, the problem of designing a controller for time-delay system \eqref{eq:sys} is discussed, i.e. the controller gain $K$becomes a variable of the LMI. Theorem \ref{sec:thm} would lead to a non-linear matrix inequality while Corollary \ref{sec:cor} gets rid of this at the price of a higher constraint on the structure of the slack variables.

Considering the average of the whole state $X$ as the output ($C = I_{n}$), the system can be written in another more useful form with :
\begin{equation}
	\left\{
		\begin{array}{ll}
			\displaystyle \dot{x}(t) = Ax(t) + \frac{1}{h} B K \int_{t-h}^t x(s) ds, & \ \ \forall t \supeq 0, \\
			x(t) = \phi(t), & \ \ \forall t \in [-h; 0],
		\end{array}
	\right.
	\label{eq:sys2}
\end{equation}
where $\phi$ is the initial condition and $x$ is the state.

The system is in the same form as the one defined in \eqref{eq:sysDelayed} with $A_D = \frac{1}{h}BK$ and $A_d = 0$. One can notice that $A_D$ depends on $K$ which is a variable in this case. The optimization based on the LMI framework cannot be applied directly because it is not a linear problem on the variable $K$. The feedback gain $K$ for a given $h$ can be found using this theorem:

\begin{theo}
	\label{sec:thmFeedback}
	Assume that, for given $h > 0$, $\varepsilon_1, \varepsilon_2, \varepsilon_3, \varepsilon_4 \in \mathbb{R}$ and $\alpha \supeq 0$, there exist matrices $P \in \mathbb{S}^{2n}_+$,  $R, S \in \mathbb{S}^n_+$, $X \in \mathbb{R}^{n \times n}$ invertible and a positive real $\beta_1$ such that the positivity LMI \eqref{eq:positivity} and the following LMI are satisfied:
	\begin{equation}
			\Phi(\alpha, h) + \text{He}\left(  \left( N \tilde{X} +  \left[ \begin{array}{cccc} 0_n & 0_n & 0_n & B\bar{K} \end{array} \right] \right)^{\top} F_{\varepsilon}  \right) \prec 0, \\	
	\end{equation}
	with the same notations than in Corollary \ref{sec:cor} but:
	 \[
	 	\begin{array}{l}
	 		N = \left[ \begin{array}{cccc} A & 0_n & -I_n & 0_n \end{array} \right], \\
	 		\tilde{X} = \text{diag}(X,X,X,X), \\
		\end{array}
	 \]
then time-delay system \eqref{eq:sys2} is $\alpha$-stable with the feedback gain $K = \bar{K} X^{-1}$.
\end{theo}

\begin{proof}
	Since $Z$ is non-singular in the proof of Corollary \ref{sec:cor}, let us introduce $X = Z^{-1}$ and $F_4 = N +  \left[ \begin{array}{cccc} 0_n & 0_n & 0_n & BK \end{array} \right]$ so that $F_4 \xi = 0$ is still valid.
	
	Multiplying on the left by $\tilde{X}^{\top}$ and on the right by $\tilde{X}$, equation \eqref{eq:LMIcor} is equivalent to the following one:
	 \begin{equation}
	 	\begin{array}{rl}
	 		\left( F_0 \tilde{X} \right)^{\top} P F_1(h) \tilde{X} + \left(F_1(h) \tilde{X} \right)^{\top} P F_0 \tilde{X} &\\
	 		+ 2 \left( \alpha F_1 \tilde{X} \right)^{\top} P F_1 \tilde{X} + \tilde{X}^{\top} \bar{S} \tilde{X} &\\
	 		 - e^{-2 \alpha h} \left( F_2 \tilde{X} \right)^{\top} \tilde{R} F_2 \tilde{X} + h^2 \left( F_3 \tilde{X} \right) ^{\top} R F_3 \tilde{X} & \\
	 		 + \text{He}\left( \tilde{X}^{\top} F_4^{\top} X^{-1} F_{\varepsilon} \tilde{X} \right) & \prec 0.
		\end{array}
		\label{eq:LMIlemma}
	 \end{equation}
	 
	 Noticing that $F_0 \tilde{X} = \bar{X} F_0$, $F_1 \tilde{X} =  \bar{X} F_1$, $F_3 \tilde{X} = X F_3$ and $F_{\varepsilon} \tilde{X} = X F_{\varepsilon}$ with $\bar{X} = \text{diag}(X,X)$, equation \eqref{eq:LMIlemma} becomes:
	 \begin{equation*}
	 	\begin{array}{rl}
	 		F_0^{\top} P_2 F_1(h) + F_1^{\top}(h) P_2 F_0 + 2 \alpha F_1^{\top} P_2 F_1 & \\
	 		- e^{-2 \alpha h} F_2 \tilde{R}_3  F_2  + h^2 F_3^{\top} R_3 F_3 &  \\
	 		+ \bar{S}_2  + \text{He}\left( \left( N \tilde{X} + \left[ \begin{array}{cccc} 0_n & 0_n & 0_n & B\bar{K} \end{array} \right] \right)^{\top} F_{\varepsilon} \right) & \prec 0,
	 	\end{array}
	\end{equation*}
	with $\bar{K} = K X$, $P_ 2= \bar{X}^{\top} P \bar{X} \succ 0$, $S_2 = X^{\top} S X$ and $R_3 = X^{\top} R X$. As $X$ is invertible, the positiveness of $P$ is equivalent to the positiveness of $P_2$ and that concludes the proof.
\end{proof}

%
%
%
%
		
\section{Observer-based Control}

Based on the preliminary section, we aim at developing an observer-based controller for time-delay system \eqref{eq:sys}.
Following the same procedure than the one described in \cite{glad2000control} for a linear time-invariant system, the estimate of $x$ will be called $\hat{x}$ and let $\varepsilon$ be $x - \hat{x}$ such that:
\begin{subequations}
\begin{align}
	\dot{\hat{x}} & = A \hat{x} + Bu + L \left( y - \frac{1}{h} C \int_{t-h}^t \hat{x}(s) ds \right) \label{eq:xHat}, \\
	\dot{\epsilon} & = A \epsilon - \frac{1}{h} L C \int_{t-h}^t \epsilon (s) ds \label{eq:obsStab},
\end{align}
\label{eq:observer}
\end{subequations}
with $L$ a $n \times p$ matrix and the others matrices are the same as before. The stability of system \eqref{eq:obsStab} leads to the convergence of $\hat{x}$ to $x$. This observer has the same structure than a Kalman filter for LTI systems but adapted to System \eqref{eq:sys}.

\subsection{Convergence of the observer}
The following theorem holds for the error system \eqref{eq:observer}:
\begin{theo} \label{th:observerStability}
	Assume that, for given $h > 0$, $\alpha \supeq 0$, $\varepsilon_1, \varepsilon_2, \varepsilon_3, \varepsilon_4 \in \mathbb{R}$, there exist a matrix $P \in \mathbb{S}^{2n}$ and $R, S \in \mathbb{S}^{2n}_+$ and a $n \times n$ invertible matrix $Z$, a $n \times p$ matrix denoted $\bar{L}$ and , $\beta_1 > 0$ such that LMI \eqref{eq:positivity} and the following LMI are satisfied:
	\begin{equation}
		\Phi(\alpha, h) + \text{He}\left( N^{\top} Z F_{\varepsilon} +  \left[ \begin{array}{cccc} 0_n & 0_n & 0_n & -\bar{L} C \end{array} \right]^{\top} F_{\varepsilon}  \right) \prec 0,
		\label{eq:lmiobsv}
	\end{equation}
	with the same notations than for Corollary \ref{sec:cor} and Theorem \ref{sec:thmFeedback} but:
	\[
	 	N = \left[ \begin{array}{cccc} A & 0_n & -I_n & 0_n \end{array} \right],
	\]
then time-delay system error $\varepsilon$ defined in \eqref{eq:obsStab} is $\alpha$-stable with the gain $L =  Z^{-\top} \bar{L}$. That means $\hat{x}$ in $\eqref{eq:xHat}$ converges exponentially to the instantaneous $x$. 
\end{theo}

\begin{proof}
	Starting from equation \eqref{eq:LMIcor} in Corollary \ref{sec:cor} with $F_4 = N + \left[ 0_{n, 3n}  \ \  -LC\right]$ so that $F_4 \xi = 0$ then:
	\[
		\Phi(\alpha, h) + \text{He}\left( \left( N + \left[ 0_{n, 3n} \ \  -LC\right] \right)^{\top} Z F_{\varepsilon} \right)  \prec 0
	 \]
	 which leads to LMI \eqref{eq:lmiobsv}
	 with $\bar{L}^{\top} = L^{\top} Z$ so $L = Z^{-T} \bar{L}$ which concludes the proof.
\end{proof}

\subsection{Feedback from reconstructed states}

Using equations \eqref{eq:observer} and $\hat{x} = x - \epsilon$, the original system can be transformed into:

\begin{equation}
	\left\{
	\begin{array}{l}
		\dot{x}(t) = (A-BK) x(t) + BK \epsilon(t), \\
		\displaystyle \dot{\epsilon}(t) = A \epsilon(t) - \frac{1}{h} LC \int_{t-h}^t \epsilon(s) ds, \\
		u(t) = -K\hat{x}(t).
	\end{array}
	\right.
	\label{eq:observerBased}
\end{equation}

Denoting $X^{\top}(t) = \left[ x^{\top}(t) \ \ \ \epsilon^{\top}(t) \right]^{\top}$ leads to:
\begin{equation}
	\dot{X}(t) = \left[ \begin{matrix} A-BK  & BK \\ 0_n & A \end{matrix} \right] X(t) + \left[ \begin{matrix} 0_n & 0_n \\ 0_n & -\frac{1}{h} LC \end{matrix} \right] \int_{t-h}^t X(s) ds.
	\label{eq:reconstructedFeedback}
\end{equation}

The following proposition gives a sufficient condition which ensures the stability of the closed loop \eqref{eq:observerBased}.

\begin{proposition}
\textbf{(Separation Principle)} The stability of the system using feedback from reconstructed states is ensured if the observer is stable and if there exists $K$ such that $A-BK$ has strictly negative eigenvalues.
\end{proposition}

\begin{proof}
	The characteristic matrix of equation \eqref{eq:reconstructedFeedback} is:
	\begin{equation*}
		\Delta(s) = \left[ \begin{matrix} sI_n - A+BK & -BK \\ 0_n & sI_n - A + \frac{1 - e^{-hs}}{hs} LC \end{matrix} \right] 
	\end{equation*}
	And its characteristic equation is:
	\begin{equation*}
			\text{det}\left( \Delta(s) \right) = \text{det}(sI_n - A + BK) \text{det} \left( sI_n - A + \tfrac{1 - e^{-hs}}{hs} LC \right) \\
	\end{equation*}
	Using $L$ as defined by Theorem \ref{th:observerStability}, and using Theorem 1.5 proposed by \cite{opac-b1100602}, $\text{det} \left( sI_n - A + \frac{1 - e^{-hs}}{hs} LC \right) = 0$ has strictly negative roots and the system \eqref{eq:reconstructedFeedback} is stable if $A-BK$ has strictly negative eigenvalues. 
\end{proof}



\section{Examples and Comparisons}

\subsection{Exponential convergence theorems}

	\subsubsection{Example 1:}
	Theorem \ref{sec:thm} and Corollary \ref{sec:cor} can ensure stability of system \eqref{eq:sysDelayed} for a given delay. A comparison of efficiency between the latter two and the theoretical bounds by \cite{Chen200795} can be done on system \eqref{eq:sysDelayed} with:
	\begin{equation}
		\begin{array}{ccccc}
			A = \left[ \begin{matrix} 0.2 & \ & 0 \\ 0.2 & & 0.1 \end{matrix}  \right]
			&
			,
			&
			A_D = \left[  \begin{matrix} -1 & \ & 0 \\ -1 & & -1 \end{matrix}  \right]
			&
			\text{and}
			&
			A_d = 0_2
		\end{array}.
		\label{eq:system1}
	\end{equation}
	Table \ref{tab:stab1} shows a comparison of the upper and lower bound for $h$ leading to a stable system using different theorems obtained with YALMIP by \cite{1393890}.
	\begin{table}[h]
		\centering
		\begin{tabular}{c|ccccc}
			& EV & Th\ref{sec:thm} & Th\ref{sec:thm} & Cor\ref{sec:cor}${}_{\varepsilon=1}$ & Cor\ref{sec:cor}${}_{\varepsilon \not =1}$ \\
			\hline
			$\alpha$ & $0$ & $0$ & $0.5$ & $0$ & $0$ \\
			\hline
			$h_{min}$ & $0.2$ & $0.2001 $ & $0.6370$ & $0.2002$ & $0.2001$ \\
			$h_{max}$ & $2.04$ & $1.9419$ & $1.0059$ & $1.8391$& $1.9108$	
		\end{tabular}
		\caption{Upper and lower bound for the delay for the system \eqref{eq:sys2} and a given decay-rate}
		\label{tab:stab1}
		\vspace{-0.6cm}
	\end{table}
	
	EV stands for eigenvalue analysis and $h_{min}$ is the lower bound of the interval for asymptotic stability while $h_{max}$ is the upper one. Results of Theorem \ref{sec:thm} are reported in Th\ref{sec:thm} for two different choices of $\alpha$. For $\alpha = 0$, this is equivalent to Theorem 6 derived by \cite{wirtinger}. Cor\ref{sec:cor}${}_{\varepsilon=1}$ stands for Corollary \ref{sec:cor} in the case of all the $\varepsilon$ equals to $1$ while Cor\ref{sec:cor}${}_{\varepsilon \not =1}$ is with $\varepsilon_1 = \varepsilon_3 = \varepsilon_4 = 1$ and $\varepsilon_2 = 0$.
	
	On this numerical example, it is possible to see the efficiency of the Wirtinger-based inequality by comparing the first and the second columns. To set an $\alpha$ different from $0$ is a very restrictive condition for the convergence and the range of feasible $h$ for $\alpha = 0.5$ is $4$ times shorter than the one for asymptotic convergence. The use of a structure for $Y$ leads to poorer results as expected. The choice of $\varepsilon_1 = \varepsilon_3 = \varepsilon_4 = 1$ and $\varepsilon_2 = 0$ is used in the examples from now on because in the examples presented in this article, it seems to give better results. As $\varepsilon_2$ is related to $A_d$ it is logical to set it to $0$.
	
	There are not so many theorems which directly deal with distributed delay systems and Figure \ref{fig:h_alpha} compares only the efficiency of Theorem \ref{sec:thm} with a pseudo-spectral analysis conducted by \cite{freq} and Corollary \ref{sec:cor} with different choices of $\varepsilon$ as explained in the previous paragraph. The gap between the pseudo-spectral analysis and Theorem \ref{sec:thm} is of a factor of nearly $2.5$ for the maximum $\alpha$ to a given $h$. Nevertheless, for small $h$ and small $\alpha$, the approximate is good and fit the pseudo-spectral curve. Possible explanations would be in the difference introduced in \eqref{eq:inequality1} and in the choice of the Lyapunov-Krasvoski functional \eqref{eq:lyap2}. The extension to Corollary \ref{sec:cor} introduces more conservatism and the choice of $\varepsilon$ has to be done carefully because it can affects the stability assessment significantly.

\begin{figure}
	\centering
	\includegraphics[width=9cm]{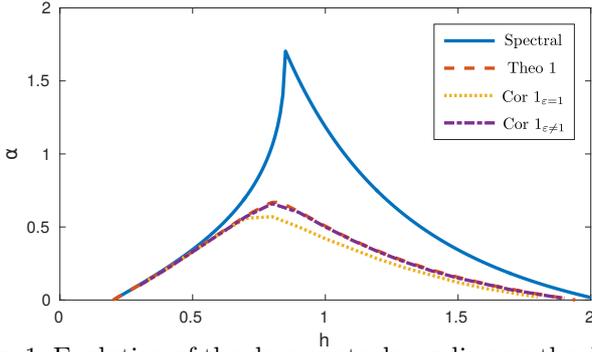}
	\vspace{-0.8cm}
	\caption{Evolution of the decay-rate depending on the delay with Theorem \ref{sec:thm} and Corollary \ref{sec:cor} for system \eqref{eq:system1}.}
	\label{fig:h_alpha}
\end{figure}

	\subsubsection{Example 2:} 
	To be compared with other results of the literature, another system with a discrete delay only is considered:
	\begin{equation}
		\begin{array}{ccccc}
			A = \left[ \begin{matrix} -3 & \ \ & -2 \\ 1 & & 0 \end{matrix}  \right]
			&
			,
			&
			A_d = \left[  \begin{matrix} -0.5 & \ \ & 0.1 \\ 0.3 & & 0 \end{matrix}  \right]
			&
			\text{and}
			&
			A_D = 0_2,
		\end{array}.
		\label{eq:system2}
	\end{equation}
Results are shown in Figure \ref{fig:h_alpha2} with the use of Theorem \ref{sec:thm}, Corollary \ref{sec:cor}, the article by \cite{mondie2005exponential} (denoted Mondie in the legend), another by \cite{lam} (denoted Xu) and stability assessment using a pseudo-spectral approach (\cite{freq}).

First of all, Theorem \ref{sec:thm} leads to good results and fit the shape of the maximum $\alpha$. The stability theorem provided by \cite{lam} gives similar results but a bit closer to the real boundary. These two theorems give a precise estimation at small $h$ which is not the case of \cite{mondie2005exponential}. Another important conclusion is the conservatism of Corollary \ref{sec:cor} compared to the others theorems for bigger $h$. The curves decrease significantly faster than the others. Nevertheless, the main interest of this corollary compared to Theorem \ref{sec:thm} is the possibility of designing controller or observer gains.

\begin{figure}
	\centering
	\includegraphics[width=9cm]{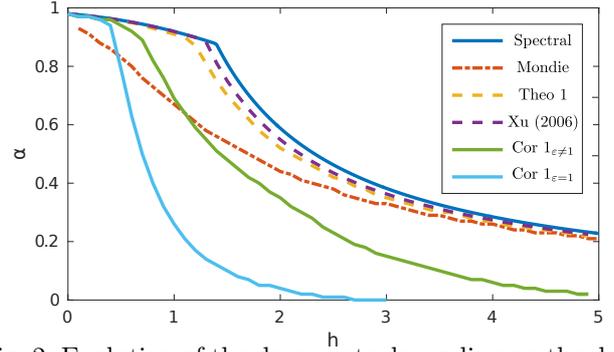}
	\vspace{-0.8cm}
	\caption{Evolution of the decay-rate depending on the delay with different theorems for system \eqref{eq:system2}.}
	\label{fig:h_alpha2}
\end{figure}

\subsection{Controller design}

	Let system \eqref{eq:sys2} defined by the matrices:
	\[
		\begin{array}{ccc}
			A = \left[ \begin{matrix} 0.2 & \ \ & 0 \\ 0.2 & & 0.1 \end{matrix}  \right]
			&
			\text{ and }
			&
			B = \left[  \begin{matrix} -1 & \ \ & 0 \\ -1 & & -1 \end{matrix}  \right]
		\end{array}.
	\]
	This system is the same than time-delay system \eqref{eq:system1}, and it is not stable without feedback for $h = 0$.
	
	\begin{table}[h]
		\centering
		\begin{tabular}{c|ccc}
			& Th\ref{sec:thmFeedback}${}_{\varepsilon \not = 1}$ & Th\ref{sec:thmFeedback}${}_{\varepsilon \not =1}$ & Th\ref{sec:thmFeedback}${}_{\varepsilon \not =1}$ \\
			\hline
			$\alpha$ & $0$ & $0.5$ & $1$  \\
			\hline
			$h_{min}$ & $0$ & $0$ & $0$ \\
			$h_{max}$ & $2.5189$ & $0.8688$ & $0.5479$
		\end{tabular}
		\caption{Upper and lower bound for the delay for the stabilized system \eqref{eq:sys2} and a given decay-rate}
		\label{tab:controller}
		\vspace{-0.7cm}
	\end{table}
	
	The simulations have been made with Matlab and 'Sdpt-3' with the Yalmip toolbox\footnote{The codes are available at \\ \textit{https://homepages.laas.fr/mbarreau/drupal/content/publications}.}.
	In Table \ref{tab:controller}, the lower bound ($h_{min}$) and the upper bound ($h_{max}$) of the delay for which there exists a matrix $K$ such that the closed-loop system is stable are summarized for different values of $\alpha$. The range of feasible delay is shrinking as the decay-rate increases. The lower bound of $h$ for the problem of stabilizing is $0$ in the examples studied which leads to the following assumption: if there exists a controller gain $K$ for a given $h > 0$ and $\alpha \supeq 0$, then System \eqref{eq:sys2} is controllable for $h=0$.

	For the observer-based control, the system to be studied has the same $A$ matrix than before and $C =  I_2$. The same $h_{max}$ and $h_{min}$ are obtained using Theorem \ref{th:observerStability} and Theorem \ref{sec:thmFeedback} for this example. Same conclusions can be drawn.

\section{Conclusion}

In this paper, we have provided a set of LMIs to assess the exponential convergence of time-delay systems using the Wirtinger-based inequality. We have also shown a comparable performance with existing theorems. However, the feature of the main result of this paper is the use for stabilization and observation of a special class of time-delay systems. An extension to non-linear systems is not straightforward but should be considered. Further work will improve the efficiency of the control and the bound for the exponential estimate by using Bessel-based inequalities. Another possible research interest would be in proving the assumption made in the last part and some robustness study on the unknown parameter $h$ for example.

\appendix

\section{Bessel-like Inequality}
\label{sec:app1}

The aim for this part is to state a Bessel-like inequality with the scalar product with $R \succ 0$:
\[
	\left \langle x_t, y_t \right \rangle = \int_{-h}^0 \int_{\theta}^0 e^{2\alpha s} x_t^{\top}(s) R y_t(s) ds d\theta.
\]

A vector $p(x_t)$ can be assimilated as the projection of $x_t$ on the function $e^{-2 \alpha \cdot }$ with the following notations:
\begin{equation}
	\begin{array}{rccl}
		p(x_t):	& [-h, 0] 	& \to 		& \mathbb{R}^n\\
				& \tau	&  \mapsto 	& \displaystyle \int_{-h}^0 \int_{\theta}^0 x_t(s) ds d\theta \frac{e^{-2\alpha \tau}}{\Xi}
	\end{array}
\end{equation}
 and $\displaystyle \Xi = \int_{-h}^0 \int_{\theta}^0 e^{-2 \alpha s} ds d\theta = \frac{e^{2 \alpha h}-2h\alpha-1}{4 \alpha^2}$.

The norm of the error is positive so that we have:
\begin{equation}
	\begin{array}{ccl}
		\langle x_t-p(x_t), x_t-p(x_t) \rangle & = & \langle x_t, x_t \rangle - 2\langle x_t, p(x_t) \rangle \\
		& & + \langle p(x_t), p(x_t) \rangle \\
		& \supeq & 0.
	 \end{array}
	 \label{eq:bessel1}
\end{equation}

Expending two terms leads to:
\begin{equation}
	\begin{array}{cl}
		\langle x_t, p(x_t) \rangle & = \displaystyle \frac{1}{\Xi} \int_{-h}^0 \int_{\theta}^0 x_t^{\top}(s) R \int_{-h}^0 \int_{\theta_1}^0 x_t(s_1) ds_1 d\theta_1 ds d\theta \\
		& = \displaystyle \frac{1}{\Xi} \int_{-h}^0 \int_{\theta}^0 x_t^{\top}(s) ds d\theta R \int_{-h}^0 \int_{\theta}^0 x_t(s) ds d\theta,
	\end{array}
	 \label{eq:bessel2}
\end{equation}
and
\begin{equation}
	\begin{array}{ccl}
		\langle p(x_t), p(x_t) \rangle & = & \displaystyle \frac{1}{\Xi^2} \int_{-h}^0 \int_{\theta}^0 \left[ \int_{-h}^0 \int_{\theta_1}^0 x_t^{\top}(s_1) ds_1 d\theta_1 \right] R \times \\
		& & \displaystyle \left[ \int_{-h}^0 \int_{\theta_1}^0 x_t(s_1) ds_1 d\theta_1\right] e^{-2 \alpha s} ds d\theta \\
		& = & \displaystyle \frac{1}{\Xi} \int_{-h}^0 \int_{\theta}^0 x_t^{\top}(s) ds d\theta R \int_{-h}^0 \int_{\theta}^0 x_t(s) ds d\theta.
	\end{array}
	 \label{eq:bessel3}
\end{equation}

Using equations \eqref{eq:bessel1}, \eqref{eq:bessel2} and \eqref{eq:bessel3} lead to:
\begin{equation}
	\begin{array}{ll}
		\displaystyle \int_{-h}^0 \int_{\theta}^0 e^{-2\alpha s} & x_t^{\top}(s) R x_t(s) ds d\theta \supeq \\
		& \displaystyle \frac{1}{\Xi} \int_{-h}^0 \int_{\theta}^0 x_t^{\top}(s) ds d\theta R \int_{-h}^0 \int_{\theta}^0 x_t(s) ds d\theta.
	\end{array}
\end{equation}

\bibliography{report_draft}

\begin{thebibliography}{4}
\providecommand{\natexlab}[1]{#1}
\providecommand{\url}[1]{\texttt{#1}}
\providecommand{\urlprefix}{URL }
\expandafter\ifx\csname urlstyle\endcsname\relax
  \providecommand{\doi}[1]{doi:\discretionary{}{}{}#1}\else
  \providecommand{\doi}{doi:\discretionary{}{}{}\begingroup
  \urlstyle{rm}\Url}\fi

\bibitem[{Able(1956)}]{Abl:56}
Able, B. (1956).
\newblock Nucleic acid content of microscope.
\newblock \emph{Nature}, 135, 7--9.

\bibitem[{Able et~al.(1954)Able, Tagg, and Rush}]{AbTaRu:54}
Able, B., Tagg, R., and Rush, M. (1954).
\newblock Enzyme-catalyzed cellular transanimations.
\newblock In A.~Round (ed.), \emph{Advances in Enzymology}, volume~2, 125--247.
  Academic Press, New York, 3rd edition.

\bibitem[{Keohane(1958)}]{Keo:58}
Keohane, R. (1958).
\newblock \emph{Power and Interdependence: World Politics in Transitions}.
\newblock Little, Brown \& Co., Boston.

\bibitem[{Powers(1985)}]{Pow:85}
Powers, T. (1985).
\newblock Is there a way out?
\newblock \emph{Harpers}, 35--47.

\end{thebibliography}


\begin{thebibliography}{23}
\providecommand{\natexlab}[1]{#1}
\providecommand{\url}[1]{\texttt{#1}}
\providecommand{\urlprefix}{URL }
\expandafter\ifx\csname urlstyle\endcsname\relax
  \providecommand{\doi}[1]{doi:\discretionary{}{}{}#1}\else
  \providecommand{\doi}{doi:\discretionary{}{}{}\begingroup
  \urlstyle{rm}\Url}\fi

\bibitem[{Boyd et~al.(1994)Boyd, {El~{G}haoui}, Feron, and Balakrishnan}]{LMI}
Boyd, S., {El~{G}haoui}, L., Feron, E., and Balakrishnan, V. (1994).
\newblock \emph{Linear Matrix Inequalities in System and Control Theory},
  volume~15 of \emph{Studies in Applied Mathematics}.
\newblock {SIAM}, Philadelphia, PA.

\bibitem[{Breda et~al.(2015)Breda, Maset, and Vermiglio}]{freq}
Breda, D., Maset, S., and Vermiglio, R. (2015).
\newblock \emph{Stability of linear delay differential equations – A
  numerical approach with MATLAB}.
\newblock SpringerBriefs in Control, Automation and Robotics T. Basar, A.
  Bicchi and M. Krstic eds. Springer.

\bibitem[{Breda(2006)}]{BREDA2006305}
Breda, D. (2006).
\newblock Solution operator approximations for characteristic roots of delay
  differential equations.
\newblock \emph{Applied Numerical Mathematics}, 56(3), 305 -- 317.

\bibitem[{Chen and Zheng(2007)}]{Chen200795}
Chen, W.H. and Zheng, W.X. (2007).
\newblock Delay-dependent robust stabilization for uncertain neutral systems
  with distributed delays.
\newblock \emph{Automatica}, 43(1), 95 -- 104.

\bibitem[{Ebihara et~al.(2015)Ebihara, Peaucelle, and Arzelier}]{Svariable}
Ebihara, Y., Peaucelle, D., and Arzelier, D. (2015).
\newblock \emph{S-Variable Approach to LMI-Based Robust Control}, volume~17 of
  \emph{Communications and Control Engineering}.
\newblock Springer.

\bibitem[{Fridman(2014)}]{norm}
Fridman, E. (2014).
\newblock \emph{Introduction to Time-Delay Systems}.
\newblock Analysis and Control. Birkh{\"a}user.

\bibitem[{Fridman et~al.(2008)Fridman, Dambrine, and
  Yeganefar}]{fridman2008input}
Fridman, E., Dambrine, M., and Yeganefar, N. (2008).
\newblock On input-to-state stability of systems with time-delay: A matrix
  inequalities approach.
\newblock \emph{Automatica}, 44(9), 2364--2369.

\bibitem[{Glad and Ljung(2000)}]{glad2000control}
Glad, T. and Ljung, L. (2000).
\newblock \emph{Control Theory}.
\newblock Control Engineering. Taylor \& Francis.

\bibitem[{Gouaisbaut and Peaucelle(2006)}]{gouaisbaut2006delay}
Gouaisbaut, F. and Peaucelle, D. (2006).
\newblock Delay-dependent stability analysis of linear time delay systems.
\newblock \emph{IFAC Proceedings Volumes}, 39(10), 54--59.

\bibitem[{Gu et~al.(2003)Gu, Kharitonov, and Chen}]{opac-b1100602}
Gu, K., Kharitonov, V.L., and Chen, J. (2003).
\newblock \emph{Stability of time-delay systems}.
\newblock Control engineering. Birkhäuser, Boston, Basel, Berlin.

\bibitem[{Kao and Rantzer(2007)}]{kao2007stability}
Kao, C.Y. and Rantzer, A. (2007).
\newblock Stability analysis of systems with uncertain time-varying delays.
\newblock \emph{Automatica}, 43(6), 959--970.

\bibitem[{Kharitonov and Zhabko(2003)}]{kharitonov2003lyapunov}
Kharitonov, V. and Zhabko, A. (2003).
\newblock Lyapunov--krasovskii approach to the robust stability analysis of
  time-delay systems.
\newblock \emph{Automatica}, 39(1), 15--20.

\bibitem[{Kolmanovskii and Myshkis(2013)}]{kolmanovskii}
Kolmanovskii, V. and Myshkis, A. (2013).
\newblock \emph{Introduction to the theory and applications of functional
  differential equations}, volume 463.
\newblock Springer Science \& Business Media.

\bibitem[{L{\"o}fberg(2004)}]{1393890}
L{\"o}fberg, J. (2004).
\newblock Yalmip : a toolbox for modeling and optimization in matlab.
\newblock \emph{Computer Aided Control Systems Design, 2004 IEEE International
  Symposium}, 284--289.

\bibitem[{Mondie and Kharitonov(2005)}]{mondie2005exponential}
Mondie, S. and Kharitonov, V. (2005).
\newblock Exponential estimates for retarded time-delay systems: an lmi
  approach.
\newblock \emph{IEEE Transactions on Automatic Control}, 50(2), 268--273.

\bibitem[{Mori et~al.(1982)Mori, Fukuma, and Kuwahara}]{mori1982estimate}
Mori, T., Fukuma, N., and Kuwahara, M. (1982).
\newblock On an estimate of the decay rate for stable linear delay systems.
\newblock \emph{International Journal of Control}, 36(1), 95--97.

\bibitem[{Olgac and Holm-Hansen(1994)}]{OLGAC199493}
Olgac, N. and Holm-Hansen, B. (1994).
\newblock A novel active vibration absorption technique: Delayed resonator.
\newblock \emph{Journal of Sound and Vibration}, 176(1), 93 -- 104.

\bibitem[{Seuret et~al.(2004)Seuret, Dambrine, and Richard}]{tds}
Seuret, A., Dambrine, M., and Richard, J.P. (2004).
\newblock Robust exponential stabilization for systems with time-varying
  delays.
\newblock \emph{TDS'04, 5th IFAC Workshop on Time Delay Systems}.

\bibitem[{Seuret and Gouaisbaut(2013)}]{wirtinger}
Seuret, A. and Gouaisbaut, F. (2013).
\newblock Wirtinger-based integral inequality: Application to time-delay
  systems.
\newblock \emph{Automatica}, 49(9), 2860 -- 2866.

\bibitem[{Seuret and Gouaisbaut(2015)}]{seuret:hal-01065142}
Seuret, A. and Gouaisbaut, F. (2015).
\newblock Hierarchy of lmi conditions for the stability analysis of time delay
  systems.
\newblock \emph{Systems and Control Letters}, 81, 1--7.

\bibitem[{Sipahi et~al.(2011)Sipahi, i.~Niculescu, Abdallah, Michiels, and
  Gu}]{5687820}
Sipahi, R., i.~Niculescu, S., Abdallah, C.T., Michiels, W., and Gu, K. (2011).
\newblock Stability and stabilization of systems with time delay.
\newblock \emph{IEEE Control Systems}, 31(1), 38--65.

\bibitem[{Trinh et~al.(2016)}]{trinh2016exponential}
Trinh, H. et~al. (2016).
\newblock Exponential stability of time-delay systems via new weighted integral
  inequalities.
\newblock \emph{Applied Mathematics and Computation}, 275, 335--344.

\bibitem[{Xu et~al.(2006)Xu, Lam, and Zhong}]{lam}
Xu, S., Lam, J., and Zhong, M. (2006).
\newblock New exponential estimates for time-delay systems.
\newblock \emph{IEEE Transactions on Automatic Control}, 51(9), 1501--1505.

\end{thebibliography}

\end{document}